\documentclass[11pt]{article}

\usepackage{amsfonts,amsmath,latexsym,color,epsfig,hyperref}
\newtheorem{theorem}{Theorem}[section]
\newtheorem{lemma}[theorem]{Lemma}

\newtheorem{corollary}[theorem]{Corollary}

\newtheorem{claim}{Claim}

\newenvironment{proof}
      {\medskip\noindent{\bf Proof:}\hspace{1mm}}
      {\hfill$\Box$\medskip}
\def\qed{\ifvmode\mbox{ }\else\unskip\fi\hskip 1em plus 10fill$\Box$}

\input{epsf}
\makeatletter
\def\Ddots{\mathinner{\mkern1mu\raise\p@
\vbox{\kern7\p@\hbox{.}}\mkern2mu
\raise4\p@\hbox{.}\mkern2mu\raise7\p@\hbox{.}\mkern1mu}}
\makeatother

\def\cp{\mbox{\rm cp}}

\title{\vspace{-0.7cm}Short proofs of some extremal results}
\author{David Conlon\thanks{Mathematical Institute, Oxford OX1 3LB, UK. Email: {\tt david.conlon@maths.ox.ac.uk}. Research supported by a Royal Society University Research 
Fellowship.} \and Jacob Fox\thanks{Department of Mathematics, MIT, Cambridge, MA 02139-4307. Email: {\tt fox@math.mit.edu}. Research supported by a Simons Fellowship and NSF grant 
DMS-1069197.} \and Benny Sudakov\thanks{Department of Mathematics, UCLA, Los Angeles, CA 90095. Email: {\tt bsudakov@math.ucla.edu}. 
Research supported in part by NSF grant DMS-1101185, by AFOSR MURI grant FA9550-10-1-0569 and by a USA-Israel BSF grant.}}
\date{}

\begin{document}
\maketitle

\begin{abstract}
We prove several results from different areas of extremal combinatorics, giving complete or partial solutions to a number of open problems. These results, coming from areas such as extremal graph theory, Ramsey theory and additive combinatorics, have been collected together because in each case the relevant proofs are quite short. 
\end{abstract}

\section{Introduction}

We study several questions from extremal combinatorics, a broad area of discrete mathematics which deals with the problem of maximizing or minimizing the cardinality of a collection of 
finite objects satisfying a certain property. The problems we consider come mainly from the areas of extremal graph theory, Ramsey theory and additive combinatorics. In each case, we 
give a complete or partial solution to an open problem posed by researchers in the area.

While each of the results in this paper is interesting in its own right, the proofs are all quite short. Accordingly, in the spirit of Alon's `Problems and results in extremal 
combinatorics' papers \cite{Al1, Al2}, we have chosen to combine them. We describe the results in brief below. For full details on each topic we refer the reader to the relevant 
section, each of which is self-contained and may be read separately from all others.

In Section \ref{sec:induced}, we improve a result of Alon \cite{Al1} on the size of the largest induced forest in a bipartite graph of given average degree. In Section 
\ref{sec:saturated}, we prove a conjecture of Balister, Lehel and Schelp \cite{BLS06} on Ramsey saturated graphs. We study the relationship between degeneracy and online Ramsey games 
in Section \ref{sec:degeneracy}, addressing a question raised by Grytczuk, Ha\l uszczak and Kierstead \cite{GHK04}. In Section \ref{sec:erdosrogers}, we improve a recent result of 
Dudek and Mubayi \cite{DM12} on generalized Ramsey numbers for hypergraphs. We prove a conjecture of Cavers and Verstra\"ete \cite{CaVe} in Section \ref{sec:partitions} by showing that 
any graph on $n$ vertices whose complement has $o(n^2)$ edges has a clique partition using $o(n^2)$ cliques. In Section \ref{sec:Hilbert}, we improve a result of Hegyv\'ari \cite{H99} 
on the size of the largest Hilbert cube that may be found in a dense subset of the integers.

All logarithms are base $2$ unless otherwise stated. For the sake of clarity of presentation, we systematically omit floor and ceiling signs whenever they are not crucial. We also do 
not make any serious attempt to optimize absolute constants in our statements and proofs.

\section{Induced forests in sparse bipartite graphs} \label{sec:induced}

Every bipartite graph trivially has an independent set with at least half of its vertices. Under what conditions can we find a considerably larger sparse set? A {\it forest} is a graph 
without cycles. Akiyama and Watanabe \cite{AkWa} and, independently, Albertson and Haas \cite{AlHa} conjectured that every planar bipartite graph on $n$ vertices contains an induced 
forest on at least $5n/8$ vertices. Motivated by this conjecture, Alon \cite{Al1} considered induced forests in sparse bipartite graphs, showing that every bipartite graph on $n$ 
vertices with average degree at most $d \geq 1$ contains an induced forest on at least $(\frac{1}{2}+e^{-bd^2})n$ vertices, for some absolute constant $b>0$. On the other hand, there 
exist bipartite graphs on $n$ vertices with average degree at most $d \geq 1$ that contain no induced forest on at least $(\frac{1}{2}+e^{-b'\sqrt{d}})n$ vertices. Alon remarks that it 
would be interesting to close the gap between the lower and upper bounds for this problem. We improve the lower bound here to $(\frac{1}{2}+d^{-bd})n$ for $d \geq 2$.

In particular, since the average degree of any planar bipartite graph is less than $4$, there is an absolute positive constant $\epsilon$ such that every planar bipartite graph on $n$ 
vertices contains an induced forest on at least $(1/2+\epsilon)n$ edges. This gives some nontrivial result on the question raised in \cite{AkWa, AlHa}. More recently, it was shown in 
\cite{KLS10} (see also \cite{Sa}) that every triangle-free (and hence every bipartite) planar graph on $n$ vertices contains an induced forest on at least $71n/128$ vertices.

As in Alon's proof, we show that every sparse bipartite graph contains a large induced subgraph whose connected components are stars.

\begin{theorem} \label{thm:induced}
Let $d$ be a positive integer. Every bipartite graph $G$ on $n$ vertices with average degree at most $d$ contains an induced subgraph on at least $(\frac{1}{2}+\delta)n$ vertices with $\delta=(2^7d^{2})^{-4d}$ whose connected components are stars.
\end{theorem}
\begin{proof}
Suppose, for contradiction, that the theorem is false. Let $X$ and $Y$ denote a bipartition of $G$ into independent sets with $|X| \leq |Y|$. We may assume $\frac{1}{2}n \leq |Y| < (\frac{1}{2}+\delta)n$ as $Y$ is an independent set and hence induces a star-forest. We have $|X| \geq (1/2-\delta)n \geq n/4$.

We will construct a sequence $Y_0 \subset Y_1 \subset \ldots \subset Y_{4d}$ of nested subsets of $Y$. Let $Y_0=\emptyset$. Once $Y_{i-1}$ has been defined, let $\delta_i=\delta+|Y_{i-1}|/n$, $d_i=1/\left(2^7d\delta_i\right)$ and $Y_i$ consist of those vertices in $Y$ with degree at least $d_i$. Note that $\delta_i$ increases, $d_i$ decreases and $Y_i$ grows as $i$ increases. As $G$ has at most $dn/2$ edges and every vertex in $Y_i$ has degree at least $d_i$, we have $|Y_i|d_i \leq d n/2$ and hence
$|Y_i|/n \leq d/(2d_i)$. We therefore have $\delta_{1}=\delta$ and 
$$\delta_{i+1} \leq \delta+d/(2d_i) =\delta+(d/2)(2^7d\delta_i)=\delta+2^6d^2\delta_i \leq 2^7d^2\delta_i.$$
Hence, for $i \geq 1$, by induction on $i$ we have $\delta_i \leq (2^7d^2)^{i-1}\delta$ and $d_i \geq 1/((2^7d^2)^i\delta)$. 

Let $e_i$ denote the number of edges containing a vertex in $Y_i \setminus Y_{i-1}$. In the claim below we will prove that $e_i \geq n/8$ for all $i \geq 1$. This will complete the proof 
by contradiction, as the number of edges of $G$ is at most $dn/2$ and at least
$$\sum_{i=1}^{4d} e_i > 4d \cdot n/8 = dn/2.$$

\begin{claim} For $i \geq 1$, $e_i > n/8$. 
\end{claim}

Indeed, suppose $e_i \leq n/8$. Let $X_i \subset X$ consist of those vertices not adjacent to any vertex in $Y_i \setminus Y_{i-1}$. We have $|X_i| \geq |X|-e_i \geq n/8$. Let $X_i' \subset X_i$ consist of those vertices of degree at most $8d$. We have $|X_i'| \geq |X_i|/2 \geq n/16$ as otherwise
$e(G) > |X_i \setminus X_i'|8d > (n/16)8d=dn/2$, a contradiction. Pick out vertices $x_1,\ldots,x_t$ from $X_i'$ greedily as follows. We pick $x_j$ arbitrarily from the remaining vertices, and delete from $X_i'$  all vertices
which share a neighbor with $x_j$ not in $Y_{i-1}$. Note that $x_j$ has degree at most $8d$ and all its neighbors not in $Y_{i-1}$ are also not in $Y_i$ (by the definition of $X_i$) and so have degree less than $d_i \geq 1$. Hence, there are at most $8d(d_i-1)$ vertices in $X_i'$ which share a neighbor with $x_j$ not in $Y_{i-1}$. We get $t \geq |X_i'|/(8dd_i) \geq n/(2^7d d_i)=\delta_i n$ and the induced subgraph of $G$ with vertex set $\{x_1,\ldots,x_t\} \cup (Y \setminus Y_{i-1})$ is a star-forest with at least $\delta_i n +|Y|-|Y_{i-1}|= \delta n+|Y| \geq (\frac{1}{2}+\delta)n$ vertices, a contradiction. This verifies the claim and completes the proof of the theorem.
\end{proof}

A much better bound may be proved for regular graphs. Indeed, it is shown in \cite{AlMuTh} that every $d$-regular bipartite graph on $n$ vertices contains an induced forest with at least $(\frac{1}{2}+\frac{1}{2(d-1)^2})n$ vertices. Moreover, this result is sharp in its dependence on $d$. It seems to us that the bound given in Theorem~\ref{thm:induced} should also be close to best possible. It would, for example, be very interesting to improve Alon's upper bound to show that there exist bipartite graphs on $n$ vertices with average degree at most $d \geq 1$ that contain no induced forest on at least $(\frac{1}{2}+e^{-b' d})n$ vertices.

\section{Ramsey saturated graphs} \label{sec:saturated}

The {\it Ramsey number} $r(G)$ of a graph $G$ is the smallest natural number $N$ such that every two-coloring of the edges of the complete graph $K_N$ contains a monochromatic copy of $G$. 
The fact that these numbers exist was first proved by Ramsey \cite{R30}.

Following Balister, Lehel and Schelp \cite{BLS06}, we say that a graph $G$ on $n$ vertices is {\it Ramsey unsaturated} if there exists an edge $e \in E(K_n)\char92 E(G)$ such that $r(G + e) 
= r(G)$. However, if $r(G + e) > r(G)$ for all edges $e \in E(K_n)\char92 E(G)$, we say that $G$ is {\it Ramsey saturated}.

There are many open questions about Ramsey saturated and unsaturated graphs. For example, it is not even known whether $K_n - e$ is saturated, that is, whether $r(K_n) > r(K_n - e)$, 
for any $n \geq 7$, though Balister, Lehel and Schelp conjecture that this should be the case.

One result proved by Balister, Lehel and Schelp \cite{BLS06} is that there are at least $\lfloor (n-2)/2 \rfloor$ non-isomorphic Ramsey saturated graphs on $n$ vertices. Moreover, they 
conjectured that there should be $c > 0$ and $\epsilon > 0$ for which there are at least $c n^{1 + \epsilon}$ non-isomorphic Ramsey saturated graphs on $n$ vertices. Here we prove this 
conjecture in a strong form, as follows.

\begin{theorem} \label{thm:saturated}
There exists $c > 0$ such that there are at least $2^{c n^2}$ non-isomorphic Ramsey saturated graphs on $n$ vertices. 
\end{theorem}

The proof of this theorem is a straightforward combination of two results from graph Ramsey theory. The first is a standard lower bound for the Ramsey number of a graph with $n$ 
vertices and $m$ edges which follows from the probabilistic method.

\begin{lemma} \label{lem:lowerRamsey}
For any graph $G$ with $n$ vertices and $m$ edges, the Ramsey number $r(G)$ satisfies
\[r(G) > 2^{\frac{m}{n} - 1}.\]
\end{lemma}

\begin{proof} Let $N = 2^{\frac{m}{n} - 1}$ and color the edges of $K_N$ at random, each edge being red or blue with probability $\frac{1}{2}$. Let $X$ be the random variable counting 
the number of monochromatic copies of $G$. Then \[\mathbb{E}[X] \leq 2^{1-m} N^n = 2^{1- m} (2^{\frac{m}{n} - 1})^n = 2^{1- n} < 1.\] It follows that there must be some coloring of 
$K_N$ which does not contain a monochromatic copy of $G$. \end{proof}

The second lemma we require is an upper bound for the Ramsey number of graphs with $n$ vertices and $m$ edges. The following result \cite{C09, FS09} is sufficient for our purposes, 
though other results \cite{C12, S11} could also be used instead. 

\begin{lemma} \label{lem:DRCRamsey}
For any bipartite graph $G$ with $n$ vertices and maximum degree $\Delta$, the Ramsey number $r(G)$ satisfies
\[r(G) \leq \Delta 2^{\Delta + 5} n.  \tag*{$\Box$}\]
\end{lemma}

{\bf Proof of Theorem \ref{thm:saturated}:} We will show, for $n$ sufficiently large, that there are at least $2^{n^2/80}$ non-isomorphic Ramsey saturated graphs. Consider a 
random labeled bipartite graph $G_1$ between two sets $A$ and $B$, each of size $n/2$, containing $\frac{1}{5} \left(\frac{n}{2}\right)^2 = \frac{n^2}{20}$ edges. By a standard large 
deviation inequality for the hypergeometric distribution, for $n$ sufficiently large, $G_1$ has maximum degree at most $n/8$ (and positive minimum degree) with probability at least 
$1/2$. That is, the number of such 
graphs is at least $\frac{1}{2} \binom{n^2/4}{n^2/20}$. For any such graph, Lemma \ref{lem:DRCRamsey} tells us that \[r(G_1) \leq 2^{\frac{n}{8} + 2} n^2 \leq 2^{\frac{n}{6}}\] for $n$ 
sufficiently large. On the other hand, by Lemma \ref{lem:lowerRamsey}, any graph $G_2$ on $A \cup  B$ with $m =\frac{n^2}{5}$ 
edges satisfies 
\[r(G_2) > 2^{\frac{n}{5} - 1} > 2^{\frac{n}{6}}.\] 
It follows that for any $G_1$ for which $r(G_1) \leq 2^{\frac{n}{6}}$ there is a graph $G$ with at most 
$\frac{n^2}{5}$ edges such that $G_1 \subseteq G$ and $G$ is Ramsey saturated. Otherwise, starting from $G_1$ and adding edges which do not increase the Ramsey number one 
by one, we could find a sequence of graphs $G_1 
\subset G_1 \cup \{e\} \subset \cdots \subset G_2$ up to a graph $G_2$ with $\frac{n^2}{5}$ edges such that $r(G_1) = r(G_1\cup\{e\}) = \cdots = r(G_2)$, contradicting 
our estimate for the Ramsey number of graphs $G_2$ with $\frac{n^2}{5}$ edges.

Since any graph $G$ with at most $\frac{n^2}{5}$ edges contains at most $\binom{n^2/5}{n^2/20}$ labeled subgraphs of size $n^2/20$, we see that the number of 
labeled Ramsey saturated graphs is at least \begin{align*} \frac{\frac{1}{2} \binom{n^2/4}{n^2/20}}{\binom{n^2/5}{n^2/20}} & = \frac{1}{2} \frac{\frac{n^2}{4} \left(\frac{n^2}{4} - 
1\right) \cdots \left(\frac{n^2}{4} - \frac{n^2}{20} + 1\right)}{\frac{n^2}{5} \left(\frac{n^2}{5} - 1\right) \cdots \left(\frac{n^2}{5} - \frac{n^2}{20} + 1\right)}\\ & \geq 
\frac{1}{2} \left(\frac{5}{4}\right)^{n^2/20} \geq 2^{n^2/70}, \end{align*} for $n$ sufficiently large. Dividing through by $n!$ tells us that the number of 
non-isomorphic Ramsey saturated graphs is, again for $n$ sufficiently large, at least $2^{n^2/80}$, as required. {\hfill$\Box$\medskip}

We note that this proof easily extends to more than $2$ colors by using the appropriate analogues of Lemmas \ref{lem:lowerRamsey} and \ref{lem:DRCRamsey}. We omit the details.

Balister, Lehel and Schelp \cite{BLS06} also conjecture that almost all graphs should be Ramsey unsaturated. While they prove that this is the case for paths and cycles of length at 
least five (a result which was recently extended by Skokan and Stein \cite{SS12}), we feel that the truth probably lies in the other direction, that is, that almost all graphs should 
be Ramsey saturated. However, it would still be very interesting to find further classes of Ramsey unsaturated graphs.

\section{Degeneracy and online Ramsey theory} \label{sec:degeneracy}

Online Ramsey games were first introduced by Beck \cite{B93} and, independently, by Friedgut, Kohayakawa, R\"odl, Ruci\'nski and Tetali \cite{FKRRT03} (see also \cite{KR05}). There are 
two players, Builder and Painter, playing on a board consisting of an infinite, independent set of vertices. At each step, Builder exposes an edge and, as each edge appears, Painter 
decides whether to color it red or blue. Builder's aim is to force Painter to draw a monochromatic copy of a fixed graph $G$.

These games have been studied from multiple perspectives. One variant asks for the least number of edges $\tilde{r}(G)$ which Builder needs to force Painter to draw a monochromatic $G$ 
(see \cite{C092} and its references). Another asks how long the game lasts if the game is played on a board with $n$ vertices and Builder chooses the edges at random (see \cite{MSS09, 
MSS092} and their references).

We will consider another variant, first introduced by Grytczuk, Ha\l uszczak and Kierstead \cite{GHK04}. Suppose that we have a positive integer-valued increasing graph parameter such 
as chromatic number, degeneracy, treewidth, thickness or genus. The question asked in \cite{GHK04} is whether, for each of these properties, there exists a function $f: \mathbb{N} 
\rightarrow \mathbb{N}$ such that Builder can force Painter to draw a monochromatic copy of any graph $G$ with parameter $k$ while himself only drawing graphs with parameter $f(k)$. To 
give a concrete example, it was proved in \cite{GHK04} (and extended in \cite{KK09} to any number of colors) that Builder may force Painter to draw a monochromatic copy of any graph 
$G$ with chromatic number $k$ while only drawing graphs of chromatic number $k$ himself.

Here we prove a similar result where the chosen parameter is degeneracy, partially addressing one of the questions raised by Grytczuk, Ha\l uszczak and Kierstead \cite{GHK04}. A graph 
is {\it $d$-degenerate} if every subgraph of it has a vertex of degree at most $d$. Equivalently, a graph $G$ is $d$-degenerate if there is an ordering of the vertices of $G$, say $u_1, u_2, 
\dots, u_n$, such that for each $1 \leq i \leq n$ the vertex $u_i$ has at most $d$ neighbors $u_j$ with $j < i$. The {\it degeneracy} of $G$ is the smallest $d$ for which $G$ is 
$d$-degenerate. We will consider the $q$-color online Ramsey game, where Painter has a choice of $q$ colors at each step, proving the following result.

\begin{theorem} \label{thm:onlinedeg}
In the $q$-color online Ramsey game, Builder may force Painter to draw a monochromatic copy of any $d$-degenerate graph while only drawing a $(qd - (q-1))$-degenerate graph.
\end{theorem}

We note that the $d = 1$ case of this theorem was already proved by Grytczuk, Ha\l uszczak and Kierstead \cite{GHK04}. Since $1$-degenerate graphs are forests, the theorem in this case states that, for any fixed number of colors, Painter may force Builder to create a monochromatic copy of any forest while only building a forest. To prove Theorem \ref{thm:onlinedeg} in the general case, we will need to use the hypergraph version of Ramsey's theorem \cite{ER52, R30}. 

\begin{lemma}
For all natural numbers $k, \ell$ and $q$ with $\ell \geq k$, there exists an integer $n$ such that if the edges of the complete $k$-uniform hypergraph on $n$ vertices are 
$q$-colored, then there is a monochromatic copy of $K_\ell^{(k)}$. \hfill $\Box$
\end{lemma}

The smallest such number $n$ is known as the {\it Ramsey number} $r_k(\ell; q)$.
\vspace{3mm}

{\bf Proof of Theorem \ref{thm:onlinedeg}:}
Suppose that $G$ is a $d$-degenerate graph with $n$ vertices. We will show that while only drawing a $(q d - (q-1))$-degenerate graph it is possible for Builder to force Painter to construct a sequence of subsets $V_0 \supset V_1 \supset \dots \supset V_t$ and a sequence of colors $C_1, \dots, C_t$, where $t = q(n-1) - (q-1)$, such that $|V_t| \geq d$ and the following holds. For every $d$-set $D$ in $V_i$ there is a vertex $v_D$ in $V_{i-1}$ such that every vertex $w \in D$ is connected to $v_D$ in color $C_i$. By the choice of $t$, the pigeonhole principle implies that there is a subsequence $V_0 = V_{i_0} \supset V_{i_1} \supset \dots \supset V_{i_{n-1}}$ such that $C_{i_1} = \dots = C_{i_{n-1}}$. Without loss of generality, we may assume that this color is red. 

It is straightforward to show that the constructed graph contains a red copy of $G$. Let $u_1, \dots, u_n$ be a $d$-degenerate ordering of the vertices of $G$, that is, such that $u_i$ has at most $d$ neighbors $u_j$ with $j < i$. We will construct an embedding $f: V(G) \rightarrow V$ of the vertices of $G$ into the red graph constructed above by induction, embedding the vertex $u_p$ into the set $V_{i_{n - p}}$. We begin by mapping $u_1$ to any vertex in $V_{i_{n-1}}$. Suppose now that $u_1, \dots, u_p$ have been embedded and we wish to embed $u_{p+1}$. We know that $u_{p+1}$ has $f \leq d$ neighbors $u_{a_1}, \dots, u_{a_f}$ with $a_j \leq p$. Since the images of each of these vertices under the embedding lie in $V_{i_{n-p}}$, we see, by taking $D = \{f(u_{a_1}), \dots, f(u_{a_f})\}$, that there is a vertex $w \in V_{i_{n-(p+1)}}$ such that the edge between $w$ and $f(u_{a_j})$ is red for all $1 \leq j \leq f$. Taking $f(u_{p+1}) = w$ completes the induction.

It remains to construct the subsets $V_0 \supset V_1 \supset \dots \supset V_t$. Let $n_t = d$ and, for each $1 \leq i \leq t$, let 
\[m_{t-i} = r_s (q d (n_{t - i + 1} + 1); q^s) \mbox{ and } n_{t - i} = m_{t-i} + \binom{m_{t-i}}{s},\]
where $s = q d - (q -1)$. We begin by taking an independent set of size $n_0$ for $V_0$. Suppose now that Builder has forced Painter to construct $V_0 \supset V_1 \supset \dots \supset V_{i-1}$ and that $V_{i-1}$ is an independent set of size $n_{i-1}$. We will show how Builder may force Painter to construct an independent set $V_i \subset V_{i-1}$ of size $n_i$ such that for every $d$-set $D$ in $V_i$ there is a vertex $v_D$ in $V_{i - 1}\char92 V_i$ such that every vertex $w \in D$ is connected to $v_D$ in a fixed color $C_{i}$. 

Suppose, therefore, that $V_{i-1}$ is an independent set of size $n_{i-1}$. By the choice of $n_{i-1}$, we may partition $V_{i-1}$ into two pieces $W_i$ and $V_{i-1} \char92 W_i$ of sizes $m_{i-1}$ and $\binom{m_{i-1}}{s}$, respectively. For each $s$-set $S$ in $W_i$, Builder now chooses a unique vertex $v_S$ in $V_{i-1}\char92 W_i$ (which is possible by the choice of size for $V_{i-1}\char92 W_i$) and joins every $w \in S$ to $v_S$. Moreover, $v_S$ will have no other neighbors in $V_{i-1}$. Note that $W_i$ is an independent set and every vertex in $V_{i-1} \char92 W_i$ has degree exactly $s = qd - (q-1)$. 

We now consider the complete $s$-uniform hypergraph on $W_i$. Suppose that the vertices of $W_i$ have been ordered. For any edge $e = \{w_1, \dots, w_s\}$ of this graph with $w_1 < \dots < w_s$, we assign it the color $(\chi(w_1 v_e), \dots, \chi(w_s v_e))$ where $v_e$ is the unique vertex in $V_{i-1} \char92 W_i$ joined to each of $w_1, \dots, w_s$ and $\chi(w_i v_e)$ is the color assigned by Painter to the edge between $w_i$ and $v_e$. This gives a $q^s$-coloring of the edges of the complete $s$-uniform hypergraph on $W_i$. By the choice of $m_{i-1}$, this set must contain a monochromatic subgraph of size $qd(n_i + 1)$. We call this set $M_i$.

For any edge $e = \{w_1, \dots, w_s\}$ in $M_i$ with $w_1 < \dots < w_s$, we know that $\chi(w_j v_e) = \chi_j$ for a fixed sequence of colors $\chi_1, \dots, \chi_s$. By the choice of $s$, we know that there must be a color $C_i$ and a set of indices $j_1, \dots, j_d$ such that $\chi(w_{j_k} v_e) = C_i$ for all edges $e$ and all $1 \leq k \leq d$. Suppose now that the vertices in $M_i$ are $w'_1 < \dots < w'_\ell$, where $\ell = q d (n_i + 1)$. Consider the subset of $M_i$ containing the vertices $w'_{qd}, w'_{2qd}, w'_{3qd}, \dots, w'_{n_i qd}$. Then, for any $1 \leq a_1 < \dots < a_d \leq n_i$, there exists an edge $e =  \{w_1, \dots, w_s\}$ of the complete $s$-uniform hypergraph on $M_i$ such that $w_{j_k} = w'_{a_k q d}$. This follows since the vertices in the subsequence are a distance $qd$ apart and we may place up to $qd - 1 \geq s$ dummy vertices between any pair (and before and after the first and the last terms in the sequence). Therefore, $\chi(w'_{a_k qd} v_e) = C_i$ for all $1 \leq k \leq d$. The set $V_i = \{w'_{qd}, w'_{2qd}, w'_{3qd}, \dots, w'_{n_i qd}\}$ has the required property that for every $d$-set $D$ in $V_i$ there is a vertex $v_D$ in $V_{i - 1}\char92 V_i$ such that every vertex $w \in D$ is connected to $v_D$ in color $C_{i}$.

To show that the graph constructed by Builder has a $(qd - (q - 1))$-degenerate ordering, we note that each vertex in $V_{i-1}\char92 W_i$ has no neighbors in $V_{i-1}\char92 W_i$ and degree $qd - (q-1)$ in $W_i$. Moreover, there are no edges between $V_i$ and $W_i\char92 V_i$. If, therefore, we choose our ordering so that for all $i = 1, 2, \dots, t$, the set of vertices in $V_i$ come before those in $W_i \char92 V_i$ and the set of vertices in $W_i \char92 V_i$ come before those in $V_{i-1} \char92 W_i$, we will have an $s$-degenerate ordering. This completes the proof.
{\hfill$\Box$\medskip}

In the other direction, it is obvious that Painter must draw a graph with degeneracy at least $d$. For $d = 1$, a matching upper bound is given by Theorem~\ref{thm:onlinedeg} and was first proved in \cite{GHK04}. It would be interesting to decide whether this simple lower bound is always sharp. It would also be interesting to know whether a similar theorem holds when degeneracy is replaced by maximum degree. This question, first studied in \cite{BGKMSW11}, appears difficult.

\section{Generalized Ramsey numbers for hypergraphs} \label{sec:erdosrogers}

Given natural numbers $s$ and $t$ with $s < t$, the {\it Erd\H{o}s--Rogers function} $f_{s,t}(n)$ is defined as the size of the largest $K_s$-free subset that may be found in any $K_t$-free graph on $n$ vertices. This function generalizes the usual Ramsey function, since determining $f_{2,t}(n)$ is the problem of determining the size of the largest independent set which is guaranteed in any $K_t$-free graph on $n$ vertices.

Since its introduction \cite{EG61, ER62} and particularly in recent years \cite{AK97, DR11, K94, K95, S051, S052}, this function has been studied quite intensively. Much of this work has focused on the case where $t = s+1$, culminating in the result \cite{DM12, DRR13, W13} that there are constants $c_1$ and $c_2$ depending only on $s$ such that
\[c_1 \left(\frac{n \log n}{\log \log n}\right)^{1/2} \leq f_{s, s+1} (n) \leq c_2 n^{1/2} (\log n)^{4 s^2}.\]
The analogous function for hypergraphs was recently studied by Dudek and Mubayi \cite{DM12}. For $s < t$, let $f_{s,t}^{(k)}(n)$ be given by
\[f_{s,t}^{(k)}(n) = \min\{\max\{|W|: W \subseteq V(\mathcal{G}) \mbox{ and $\mathcal{G}[W]$ contains no $K_s^{(k)}$}\}\},\]
where the minimum is taken over all $K_t^{(k)}$-free $k$-uniform hypergraphs $\mathcal{G}$ on $n$ vertices. Dudek and Mubayi proved that, for $k = 3$ and $3 \leq s <t$, this function satisfies
\[f_{s-1,t-1}^{(2)}(\lfloor \sqrt{\log n} \rfloor) \leq f_{s, t}^{(3)}(n) \leq c_s \log n.\]
In particular, for $t = s+1$, this gives constants $c_1$ and $c_2$ depending only on $s$ such that
\[c_1 (\log n)^{1/4} \left(\frac{\log\log n}{\log\log\log n}\right)^{1/2} \leq f_{s, s+1}^{(3)}(n) \leq c_2 \log n.\]
Here we make an improvement to the lower bound, using ideas on hypergraph Ramsey numbers developed by the authors in \cite{CFS10}.

We will need the following lemma, due to Shearer \cite{Sh95}, which gives an estimate for the size of the largest independent set in a sparse $K_s$-free graph.

\begin{lemma} \label{lem:Shearer}
There exists a constant $c_s$ such that if $G$ is a $K_s$-free graph on $n$ vertices with average degree at most $d$ then $G$ contains an independent set of size at least 
\[c_s \frac{\log d}{d \log \log d} n. \tag*{$\Box$}\]
\end{lemma}

The main result of this section is now as follows.

\begin{theorem}
For any natural number $s \geq 3$, there exists a constant $c$ such that
\[f_{s, s+1}^{(3)} (n) \geq c \left(\frac{\log n}{\log \log \log n}\right)^{1/3}.\]
\end{theorem}

\begin{proof}
Let $\mathcal{G}$ be a $K_{s+1}^{(3)}$-free $3$-uniform hypergraph on $n$ vertices. Let 
\[p = c \left(\frac{\log n}{\log \log \log n}\right)^{1/3},\] 
where $c$ is a constant to be determined later. Our aim will be to show that $\mathcal{G}$ contains a $K_s^{(3)}$-free subgraph of size at least $p$. 

We will construct, by induction, a sequence of vertices $v_1, v_2, \dots, v_\ell$ and a non-empty set $V_{\ell}$ such that, for any $1 \leq i < j \leq \ell$, all triples $\{v_i, v_j, v_k\}$ with $j < k \leq \ell$ and all triples $\{v_i, v_j, w\}$ with $w \in V_{\ell}$ are either all edges of $\mathcal{G}$ or all not edges of $\mathcal{G}$.  At each step, we consider the auxiliary graph $G_{\ell}$ on vertex set $\{v_1, \dots, v_{\ell}\}$ formed by connecting $v_i$ and $v_j$ with $i < j$ if and only if all triples $\{v_i, v_j, v_k\}$ with $j < k \leq \ell$ and all triples $\{v_i, v_j, w\}$ with $w \in V_{\ell}$ are edges of $\mathcal{G}$.  We note that $G_{\ell}$ must be $K_s$-free. Otherwise, if $\{v_{i_1}, \dots, v_{i_s}\}$ are the vertices of a $K_s$, we may take any vertex $w$ in $V_{\ell}$ to form a $K_{s+1}^{(3)}$ in $\mathcal{G}$, namely, $\{v_{i_1}, \dots, v_{i_s}, w\}$.

If at any point $G_{\ell}$ contains a vertex $u$ with at least $p$ neighbors, we stop the process. Since $G_{\ell}$ is $K_s$-free, it follows that the neighborhood $U$ of $u$ in $G_{\ell}$ is $K_{s-1}$-free. The set $U$ must also be $K_s^{(3)}$-free in $\mathcal{G}$. Indeed, suppose otherwise and that $\{v_{i_1}, \dots, v_{i_s}\}$ with $i_1 < \dots < i_s$ is a $K_s^{(3)}$ with vertices from $U$. Then, by the construction of $G_{\ell}$, the set $\{v_{i_1}, \dots, v_{i_{s-1}}\}$ must be a $K_{s-1}$ in $U$. We may therefore assume that the maximum degree in each $G_{\ell}$ is at most $p$. 

To begin our induction, we fix $v_1$ and let $V_1 = V(\mathcal{G})\char92 \{v_1\}$. Suppose now that we have constructed the sequence $v_1, v_2, \dots, v_{\ell}$ and a set $V_{\ell}$ satisfying the required conditions and we wish to find a vertex $v_{\ell + 1}$ and a set $V_{\ell + 1}$. We let $v_{\ell + 1}$ be any vertex from the set $V_{\ell}$. Let $V_{\ell,0} = V_\ell \char92 \{v_{\ell + 1}\}$. We will construct a sequence of subsets $V_{\ell,0} \supset V_{\ell, 1} \supset \dots \supset V_{\ell, \ell}$ such that all triples $\{v_i, v_{\ell + 1}, w\}$ with $1 \leq i \leq j$ and $w \in V_{\ell, j}$ are either all edges of $\mathcal{G}$ or all not edges of $\mathcal{G}$, depending only on the value of $i$. (Note that since each $V_{\ell, j} \subset V_{\ell}$ it follows that all triples $\{v_i, v_j, w\}$ with $1 \leq i < j \leq \ell$ and $w \in V_{\ell, j}$ are either all edges of $\mathcal{G}$ or all not edges of $\mathcal{G}$, depending only on the values of $i$ and $j$.)

Suppose then that $V_{\ell, j}$ has been constructed in an appropriate fashion. To construct $V_{\ell, j+1}$, we consider the neighborhood of the vertices $v_{j+1}$ and $v_{\ell + 1}$ in $V_{\ell, j}$. If this neighborhood has size at least $\alpha |V_{\ell, j}|$, we let $V_{\ell, j+1}$ be this neighborhood. Otherwise, we let $V_{\ell, j+1}$ be the complement of this neighborhood in $V_{\ell, j}$. Note that in this case $|V_{\ell, j+1}| \geq (1 - \alpha) |V_{\ell, j}|$. To finish the construction of $V_{\ell + 1}$, we let $V_{\ell + 1} = V_{\ell, \ell}$. It is easy to check that it satisfies the required conditions.  

We halt the process when $\ell = m$. Recall that the maximum degree of $G_m$ is at most $p$. Since $G_m$ is also $K_s$-free, Lemma \ref{lem:Shearer} implies that there is a constant $c_s$ such that the graph contains an independent set of size at least 
\[c_s \frac{\log p}{p \log \log p} m = p,\]
by choosing $m = \frac{1}{c_s} \frac{\log\log p}{\log p} p^2$. This in turn implies an independent set of size $p$ in $\mathcal{G}$, completing the proof provided only that $|V_m| \geq 1$.  To verify this, note that if in $G_i$ the vertex $v_i$ has degree $d_i$ then
\[|V_i|  \geq \alpha^{d_i} (1 - \alpha)^{i - 1 - d_i} (|V_{i-1}| - 1) \geq \alpha^{d_i} (1 - \alpha)^{i - 1 - d_i} |V_{i-1}| - 1.\]
Telescoping over all $i = 2, \dots, m$, it follows that, for $\alpha \leq \frac{1}{2}$,
\begin{eqnarray*}
|V_m| & \geq & \alpha^{\sum_{i=2}^m d_i} (1 - \alpha)^{\binom{m}{2} - \sum_{i=2}^m d_i} n - m\\
& \geq & \alpha^{\frac{pm}{2}} (1 - \alpha)^{\binom{m}{2} - \frac{pm}{2}} n - m.
\end{eqnarray*}
The second inequality follows by noting that $\sum_{i=2}^m d_i = e(G_m) \leq \frac{pm}{2}$ and observing that the function $\alpha^t (1- \alpha)^{\binom{m}{2} - t}$ is decreasing in $t$ for $\alpha \leq\frac{1}{2}$. Therefore, taking $\alpha = \frac{p}{m} \log(\frac{m}{p})$ (note that $\alpha \leq \frac{1}{2}$ for $n$ sufficiently large) and using that $2^{-2x} \leq 1 - x$ for $0 \leq x \leq \frac{1}{2}$, 
\begin{eqnarray*}
|V_m| & \geq & \alpha^{\frac{pm}{2}} (1 - \alpha)^{\frac{m^2}{2}} n - m \geq (p/m)^{\frac{pm}{2}} 2^{-\alpha m^2} n - m\\
& = & 2^{-\frac{3}{2} pm \log (m/p)} n - m \geq \sqrt{n} - m \geq 1.
\end{eqnarray*}
In the fourth inequality, we used that for $n$ sufficiently large depending on $c_s$, $\log(m/p) \leq \log p$ and, therefore, 
\[\frac{3}{2}pm \log (m/p) \leq \frac{3}{2c_s} p^3 \frac{\log \log p}{\log p} \log p = \frac{3}{2c_s} p^3 \log \log p.\]
Hence, since $p = c \left(\frac{\log n}{\log \log \log n}\right)^{1/3}$, for $c$ sufficiently small we have
\[\frac{3}{2}pm \log (m/p) \leq \frac{1}{2} \log n.\]
This completes the proof.
\end{proof}

This result easily extends to higher uniformities to give that $f_{s, s+1}^{(k)} (n) \geq (\log_{(k-2)} n)^{1/3 - o(1)}.$
Here $\log_{(0)} x = x$ and $\log_{(i+1)} x = \log (\log_{(i)} x)$. This improves an analogous result of Dudek and Mubayi~\cite{DM12} with a $1/4$ in the exponent but remains far from their upper bound $f_{s, s+1}^{(k)} (n) \leq c_{s,k} (\log n)^{1/(k-2)}$. It is an interesting open question to close the gap between the upper and lower bounds.

\section{Clique partitions of very dense graphs} \label{sec:partitions}

A {\it clique partition} of a graph $G$ is a collection of complete subgraphs of $G$ that partition the edge set of $G$.
The {\it clique partition number} $\cp(G)$ is the smallest number of cliques in a clique partition of $G$.  Despite receiving considerable attention over the last 60 years, this interesting graph parameter is still not well understood.

One class of graphs for which this parameter has been studied quite extensively is when the graph $G$ is the complement $\overline{F}$ of a sparse graph $F$. Orlin \cite{Or} was the first to ask about the asymptotics of the clique partition number for the complement of a perfect matching. If $F$ is a perfect matching on $n$ vertices, Wallis \cite{Wa} showed that $\cp(\overline{F}) = o(n^{1+\epsilon})$ for any $\epsilon>0$. This was later improved by Gregory, McGuinness and Wallis~\cite{GrMcWa} to  $\cp(\overline{F}) = O(n\log \log n)$. Wallis \cite{Wa1} later showed that the same bound holds if $F$ is a Hamiltonian path on $n$ vertices. For any forest $F$ on $n$ vertices, Cavers and Verstra\"ete \cite{CaVe} proved that $\cp(\overline{F})=O(n\log n)$. They also conjectured that there are forests $F$ on $n$ vertices such that $\cp(\overline{F})$ grows superlinearly in $n$. Here we extend their proof to give the following improvement.

\begin{theorem}\label{for}
If $F$ is a forest on $n$ vertices, then $\cp(\overline{F}) = O(n\log \log n)$.
\end{theorem}

A {\it Steiner $(n,k)$-system} is a family of $k$-element subsets of an $n$-element set such that each pair of elements appears in exactly one of the subsets. One can view a Steiner $(n,k)$-system as a clique partition of $K_n$ into cliques of size $k$. It is an open problem to determine for which pairs $(n,k)$ a Steiner $(n,k)$-system exists. Necessary conditions for the existence of a Steiner $(n,k)$-system are $n \equiv 1$ mod $k-1$ and $n(n-1) \equiv 0$ mod $k(k-1)$. Wilson \cite{Wil74} showed that for each $k$ there is $n(k)$ such that the necessary conditions for the existence of a Steiner $(n,k)$-system are sufficient provided that $n \geq n(k)$. An explicit upper bound on $n(k)$ which is triple exponential in $k^2\log k$ was proved by 
Chang \cite{Cha}. 

For sparse graphs $F$, Cavers and Verstra\"ete \cite{CaVe} proved an upper bound on the clique partition number $\cp(\overline{F})$ which is conditional on the existence of certain Steiner systems. They prove that if $F$ is a graph on $n$ vertices with maximum degree $\Delta = o(n/\log^4 n)$ and there is a Steiner $(n,k)$-system with $k=\lfloor \left(\frac{n}{2\Delta}\right)^{1/2} \rfloor$ then $\cp(\overline{F})=O\left(n^{3/2}\Delta^{1/2}\log^2 n\right)$. They also conjectured that if the maximum degree of $F$ is $o(n)$ then $\cp(\overline{F})=o(n^2)$. This conjecture follows from the next theorem, which gives an unconditional improvement on their bound. 

\begin{theorem}\label{maintheorem}
If $F$ has $n$ vertices and $m \geq \sqrt{n}$ edges, then $\cp(\overline{F}) = O\left((mn)^{2/3}\right).$
\end{theorem}

We have the following immediate corollary, showing that we can partition the complement of a sparse graph into few cliques.

\begin{corollary}
If $F$ has $n$ vertices and $o(n^2)$ edges, then $\cp(\overline{F}) = o(n^2)$.
\end{corollary}

To prove Theorem \ref{maintheorem}, we will first prove a useful lemma saying that, for $2 \leq k < n$, we can partition the complete graph $K_n$ into $O(\max((n/k)^2,n))$ cliques of order at most $k$. We begin with some elementary observations. For graphs $G$ and $H$, let $G \cup H$ denote the disjoint union of $G$ and $H$. 

\begin{lemma}\label{simpleprop}
For $s,t,n$ with $s+t \leq n$, we have 
$\cp(K_n \setminus (K_s \cup K_t)) \geq st$.
\end{lemma}
\begin{proof}
Indeed, in a clique-partition of $K_n \setminus (K_s \cup K_t)$, each of the $st$ edges between the independent set of size $s$ and the independent set of size $t$ must be in different cliques.
\end{proof}

\begin{lemma}\label{obvious}
If $G$ has $n$ vertices and all but $t \geq 1$ vertices of $G$ have degree $n-1$, then
$\cp(G) < tn$.
\end{lemma}
\begin{proof} 
The clique partition of $G$ consisting of the clique on the $n-t$ vertices of degree $n-1$ and a $K_2$ for each remaining edge uses at most 
$1+t(n-2) < tn$ cliques. 
\end{proof} 

In order to prove our main auxiliary lemma, we need to know a little about the particular class of Steiner systems known as projective planes. A {\it projective plane} consists of a set of points and a set of lines and a relation between points and lines called incidence having the following properties:
\begin{itemize}
\item For any two distinct points, there is exactly one line incident to both of them, 
\item For any two distinct lines, there is exactly one point incident to both of them,
\item There are four points such that no line is incident with more than two of them. 
\end{itemize}
The first condition says that any two points determine a line, the second condition says that any two lines intersect in one point and the 
last condition excludes some degenerate cases. It can be shown that a projective plane has the same number of lines as it has points. Any finite projective plane has $q^2+q+1$ points for some positive integer $q$ and we denote such a projective plane by $P_q$. Each line is incident with $q+1$ points and each point is incident with $q+1$ lines. Therefore, the lines of $P_q$ form a Steiner $(q^2+q+1,q+1)$-system. It is known that if $q$ is a prime power then there is a projective plane on $q^2+q+1$ points. It is a famous open problem to determine if there exist finite projective planes of other orders. 

\begin{lemma}\label{useful}
Let $k \geq 2$ and $n$ be positive integers and let $f(n,k)$ denote the minimum number of cliques each on at most $k$ vertices needed to clique partition $K_n$. If $n \leq k$, trivially $f(n,k)=1$. If $n > k$, then
$$f(n,k)=\Theta\left(\max\left((n/k)^2,n\right)\right).$$
\end{lemma}
\begin{proof}
We first prove the lower bound. The pigeonhole principle implies the lower bound $f(n,k) \geq {n \choose 2}/{k \choose 2} \geq (n/k)^2$ as we need to cover ${n \choose 2}$ edges by cliques each with at most ${k \choose 2}$ edges. Now suppose $\sqrt{n} \leq k < n$ and we have a partition of the edges of $K_n$ into cliques each with at most $k$ vertices. Let $s \geq t$ be the sizes of the two largest cliques used in the partition, so $t \geq 2$. If $s \leq 2 \sqrt{n}$, then the number of cliques used is at least $f(n,s) \geq n^2/s^2 \geq n/4$. So suppose $s > 2\sqrt{n}$. By Lemma \ref{simpleprop}, there are at least $(s-1)(t-1)$ remaining cliques in the partition (the cliques may intersect in one vertex). So suppose $(s-1)(t-1) < n/4$. Then $st < n$. Since all cliques besides the largest have at most $t$ vertices the number of cliques used is at least $$1+\frac{{n \choose 2}-{s \choose 2}}{{t \choose 2}} = 1+\frac{(n-s)(n+s-1)}{t(t-1)} > 1+\frac{(n-s)(n+s-1)}{(n/s)((n/s)-1)} = 1+\frac{s^2(n+s-1)}{n} \geq 1+s^2 > n$$ as $\sqrt{n} \leq s$ and $t <n/s$. We have thus proved that $f(n,k) \geq \max((n/k)^2,n/4)$.

We now turn to proving the upper bound. We will prove by induction on $n$ that the bound $f(n,k) \leq \max(200(n/k)^2,4n)$ holds. We may assume that $k > 20$ as otherwise we can clique partition $K_n$ into ${n \choose 2} \leq 200(n/k)^2$ edges. We will use Bertrand's postulate that there is a prime between $x$ and $2x$ for every $x \geq 1$. Much better estimates are known on the distribution of the primes, but this is sufficient for our purposes.

If $k \geq 2\sqrt{n}$, let $q$ be the smallest prime such that  $q^2+q+1 \geq n$. Bertrand's postulate implies that $q^2+q+1 < 4n$, and hence $q < k$. Consider a projective plane $P_q$ on $q^2+q+1$ points and edge-partition the complete graph on these points into cliques of size $q+1 \leq k$ given by the lines of $P_q$. There are $q^2+q+1 < 4n$ such cliques and, restricting to $n$ of these points, we get in this case $f(n,k) \leq q^2+q+1 < 4n$.

If $k < 2\sqrt{n}$, let $q$ be the smallest prime which is at least $n/k+k$. Bertrand's postulate implies that $q \leq 2(n/k+k) < 2(n/k+2\sqrt{n})<10n/k$. The Tur\'an graph $T_{n,k}$ is the complete $k$-partite graph on $n$ vertices with parts of size as equal as possible. We will find a partition of the edges of $T_{n,k}$ into cliques of size at most $k$ using at most $q^2+q-k+1$ cliques. Let $S_1,\ldots,S_k$ denote the $k$ independent sets of $T_{n,k}$. Let $L_1,L_2,\ldots,L_k$ denote $k$ lines of the projective plane $P_q$. View each $S_i$ as a subset of points of the line $L_i$ such that the points of $S_1,\ldots,S_k$ do not include any of the intersection points of $L_1,\ldots,L_k$. We can do this since $|S_i|\leq \lceil n/k \rceil$, $|L_i|=q+1\geq n/k+k+1$, each pair of lines have one intersection point, and so each $L_i$ has precisely $k-1$ intersection points in total with all the other $L_j$. Every pair of points in $S_1 \cup \ldots \cup S_k$ other than those inside one of the $S_i$ is contained in exactly one of the $q^2+q-k+1$ lines of $P_q$ other than $L_1,\ldots,L_k$. Since each of these lines intersects $L_i$ and hence $S_i$ in at most one point, these $q^2+q-k+1$ lines give an edge-partition of $T_{n,k}$ into $q^2+q-k+1$ cliques each with at most $k$ vertices. This gives the bound
\begin{eqnarray*} 
f(n,k) & \leq & q^2+q-k+1+kf(\lceil n/k \rceil,k) \leq 110(n/k)^2+k\max(200(\lceil n/k \rceil)^2/k^2,4 \lceil n/k \rceil) \\ & \leq &  110(n/k)^2+800n^2/k^3+8n \leq 200(n/k)^2,
\end{eqnarray*}
where we use $20 < k < 2\sqrt{n}$ and that, for each $k$, $f(n,k)$ is a monotonically increasing function of $n$. This latter fact follows by noting that restricting a clique partition of a clique to a subclique is a clique partition of the subclique. This completes the proof by induction on $n$. 
\end{proof}

One can easily modify the above proof, using the prime number theorem instead of Bertrand's postulate, to show that 
$f(n,k)=(1+o(1))\frac{n^2}{k(k-1)}$ as long as $k=o(\sqrt{n})$. 
\vspace{3mm}

{\bf Proof of Theorem \ref{maintheorem}:} 
Let $F$ be a graph with $n$ vertices and $m \geq \sqrt{n}$ edges and $k=(n^2/m)^{1/3}\leq \sqrt{n}$. By Lemma \ref{useful}, we can partition the complete graph on $n$ vertices into 
$N=O\left((n/k)^2\right)$ cliques $Q_1,\ldots,Q_N$ each of order at most $k$. 
For each clique $Q_i$ let $t_i$ be the number of edges of $F$ it contains. 
Each $Q_i$ that contains no edge of $F$ will be used in the clique partition of $\overline{F}$.
If $t_i \geq 1$, we use Lemma \ref{obvious} to partition the induced subgraph of $\overline{F}$ on the vertex set of 
$Q_i$ into $2t_i|Q_i| \leq 2t_i k$ cliques. Thus, we get that $\overline{F}$ can 
be partitioned into at most $N+2\sum_i t_ik=N+2mk=O\left((mn)^{2/3}\right)$ cliques, which completes the proof of Theorem \ref{maintheorem}. 
\qed \vspace{3mm}

Note that adding an edge to a graph can increase the clique partition number by at most $1$. It follows that for any forest $F$, if $T$ is a spanning tree containing $F$ then $\cp(\overline{F}) \leq \cp(\overline{T})+n-1$. So, to prove Theorem \ref{for}, it is sufficient to prove it for trees. Let $g(n)$ denote the maximum of $\cp(\overline{T})$ over all trees $T$ on $n \geq 2$ vertices. We will prove by induction that 
$$g(n) \leq 100n(1+\log \log n),$$ 
which verifies Theorem \ref{for}. This inequality clearly holds for $n \leq 200$ as $g(n) \leq {n \choose 2} \leq 100n$ in this case. So suppose $n >200$.

A {\it tree partition} of a graph $G$ is a collection $\{T_1,\ldots,T_r\}$ of subtrees of $G$ such that each edge of $G$ is in exactly one tree and, 
for all $i \not = j$, $T_i$ and $T_j$ share at most one vertex in common. We will use the following simple lemma from \cite{CaVe}. 

\begin{lemma} Let $T$ be a tree on $n$ vertices and $2 \leq k \leq n$. Then there exists a tree partition 
$\{T_1,\ldots,T_r\}$ of $T$ into at most $2n/k$ trees such that the number of vertices of each $T_i$ is between $k/3$ and $k$. \hfill $\Box$
\end{lemma}

{\bf Proof of Theorem \ref{for}:} Let $k=\sqrt{n}$ and apply the above lemma to find a tree partition of $T$ into $r \leq 2n/k = 2\sqrt{n}$ trees each with at 
most $k=\sqrt{n}$ vertices. Order these trees $\{T_1,\ldots,T_r\}$ so that for $i \geq 2$ the union of $\bigcup_{j=1}^{i-1} V(T_j)$ is connected. Then
there is exactly one vertex $v_i$ of $T_i$ that is contained in $\bigcup_{j=1}^{i-1} V(T_j)$. By 
Bertrand's postulate, there is a prime $q$ satisfying $3\sqrt{n} \leq q \leq 6\sqrt{n}$. We will show that there is a one-to-one mapping of the vertices of $T$ into the points of the 
projective plane $P_q$ on $q^2+q+1$ points such that each $T_i$ is contained in a line $L_i$ and no vertices of $T$ other than those in $T_i$ map to a point in $L_i$.

Let $L_1$ be an arbitrary line in the projective plane $P_q$ on $q^2+q+1$ vertices. Arbitrarily embed the vertices of $T_1$ into the points of $L_1$ with the vertex $v_2$ identified 
with some point $w_2$. This can be done since $|T_1| \leq k < q+1 = |L_1|$. Suppose we have already embedded $L_1,\ldots,L_{i-1}$ and let $w_i$ denote the image of the vertex $v_i$ 
which is in at least one line $L_j$ with $j < i$. Pick an arbitrary line $L_i$ through $w_i$ with $L_i \not = L_j$ for $j < i$. Since there are $q+1>r>i-1$ lines through each point, 
and, in particular, through $w_i$, we can indeed pick such a line $L_i$. Arbitrarily embed the remaining vertices $V(T_i) \setminus \{v_i\}$ of $T_i$ amongst the points of $L_i$ not in 
any $L_j$ with $j<i$. Since any two lines intersect in exactly one point $L_i$ has $q+1-(i-1)\geq q+2-r>k \geq |T_i|$ points not in any $L_j$ with $j<i$. Thus we can indeed embed 
these remaining vertices. This demonstrates that we can find the desired mapping of the vertices of $T$ into the points of $P_q$.

We next construct a clique partition of $\overline{T}$. For each line of the projective plane not containing an edge of $T$, we use the corresponding clique restricted to the vertices 
of $T$. For each line of the projective plane containing an edge of $T$, this line contains the vertices of $T_i$ and no other vertex of $T$, so we use at most $g(|T_i|)$ cliques to 
partition the edges of $\overline{T_i}$. Note that $\sum_{i=1}^r |T_i|=|T|+r-1 \leq n+2n/k-1<n+2\sqrt{n}$. Using this estimate, together with the 
inequalities $|T_i| \leq \sqrt{n}$, $n > 200$ and the induction hypothesis, we get 
\begin{eqnarray*} g(n) & \leq & q^2+q+1+\sum_{i=1}^r g(|T_i|) \leq 45n+100(n+2\sqrt{n})(1+\log \log \sqrt{n}) \\  & = & 45n+100(n+2\sqrt{n})\log 
\log n \leq 100n(1+\log \log n),\end{eqnarray*} as required.\qed

\section{Hilbert cubes in dense sets} \label{sec:Hilbert}

A {\it Hilbert cube} or an affine cube is a set $H \subset \mathbb{N}$ of the form
\[H = H(x_0, x_1, \dots, x_d) = \left\{x_0 + \sum_{i \in I} x_i: I \subseteq [d]\right\},\]
where $x_0$ is a non-negative integer and $x_1, \dots, x_d$ are positive integers. We will always assume here that the generators $x_1,\ldots, x_d$ are all distinct. We refer to the index $d$ as the dimension.

One of the earliest results in Ramsey theory is a theorem of Hilbert \cite{H1892} stating that if $n$ is sufficiently large then any coloring of the set $[n]$ with a fixed number of colors, say $r$, must contain a monochromatic Hilbert cube of dimension $d$. The smallest such $n$ we denote by $h(d, r)$. The best known upper bound for this function is 
\[h(d,r) \leq (2r)^{2^{d-1}}.\]
As noted in \cite{BCEG85}, a double exponential upper bound already follows from Hilbert's original argument. The result we have quoted is a slight strengthening noted by Gunderson, R\"odl and Sidorenko \cite{GRS99} which follows from a stronger density statement. 

This density version \cite{Sz69} states that for any natural number $d$ and $\delta > 0$ there exists an $n_0$ such that if $n \geq n_0$ any subset of $[n]$ containing $\delta n$ elements contains a Hilbert cube of dimension $d$. Quantitative versions of this lemma imply that any subset of $[n]$ of fixed positive density $\delta$ contains a Hilbert cube of dimension at least $c' \log \log n$, where $c'$ depends only on $\delta$. This in turn gives a double exponential upper bound for the original coloring problem. 

On the other hand, Hegyv\'ari \cite{H99} gave a lower bound by proving that with high probability a random subset of $[n]$ of small but fixed positive density $\delta$ does not contain Hilbert cubes of dimension $c \sqrt{\log n \log \log n}$, where $c$ depends only on $\delta$. Here we improve this result as follows.

\begin{theorem} \label{thm:Hegy}
For any $0 < \delta < 1$, there exists $c > 0$ such that with high probability a random subset of the set $[n]$, where each element is chosen independently with probability $\delta$, does not contain Hilbert cubes of dimension $c \sqrt{\log n}$.   
\end{theorem}

By another result of Hegyv\'ari \cite{H99}, this theorem is sharp up to the constant $c$. That is, with high probability, dense random subsets of $[n]$ contain Hilbert cubes of dimension $c' \sqrt{\log n}$ for some $c'$ depending on the density.

Let $X$ be a set with elements $1 \leq x_1 < x_2 < \dots < x_d$ and write
\[\Sigma(X) = \left\{\sum_{i \in I} x_i: I \subseteq [d]\right\}.\]
Note that a Hilbert cube is just a translation of an appropriate $\Sigma(X)$. The following basic estimates for $|\Sigma(X)|$ will be useful to us. For a proof, see \cite{H96}. 

\begin{lemma} \label{lem:Sigmabasic}
For any set $X$ with $d$ elements,
\[{d+1 \choose 2}+1 \leq |\Sigma(X)| \leq 2^d. \tag*{$\Box$}\]
\end{lemma}

The main new ingredient that we use is the following inverse Littlewood--Offord theorem due to Tao and Vu \cite{TV09} (see also \cite {NV11} and \cite{TV10} for improved versions). This says that if $|\Sigma(X)|$ is small then $X$ must be highly structured. A {\it generalized arithmetic progression} or GAP for short is a subset $Q$ of $\mathbb{Z}$ of the form 
\[Q = \left\{x_0 + a_1 x_1 + \dots + a_r x_r: 1 \leq a_i \leq A_i\right\}.\]
We refer to $r$ as the rank of the GAP $Q$ and the product $A_1 \dots A_r$ as the volume of $Q$. Rephrased in our terms, the inverse Littlewood--Offord theorem states that if $|\Sigma(X)|$ is small then almost all of $X$ must be contained in a GAP $Q$ of low rank.

\begin{lemma} \label{lem:ILO}
For every $C > 0$ and $0 < \epsilon < 1$, there exist positive constants $r$ and $C'$ such that if $X$ is a multiset with $d$ elements and $|\Sigma(X)| \leq d^C$ then there is a GAP $Q$ of dimension $r$ and volume at most $d^{C'}$ such that all but at most $d^{1 - \epsilon}$ elements of $X$ are contained in $Q$. \hfill $\Box$
\end{lemma}

The following lemma is the key step in our proof.

\begin{lemma}\label{hilbertcount} For $s \leq \log d$, the number of $d$-sets $X \subset [n]$ with $|\Sigma(X)| \leq 2^sd^2$ is at most $n^{O(s)}d^{O(d)}$. 
\end{lemma}
\begin{proof}
Let $C=3$, $\epsilon=1/2$ and $r$ and $C'$ be the constants given by Lemma \ref{lem:ILO}. If a $d$-set $X$ satisfies $|\Sigma(X)| \leq 2^sd^2 \leq d^3$, then, by Lemma \ref{lem:ILO}, it has a subset $X' \subset X$ which lies in a  GAP $Q$ of dimension $r$ and volume at most $d^{C'}$ such that $|X \setminus X'| \leq d^{1-\epsilon}=d^{1/2}$. 

The number of possible choices for $Q$ is $d^{O(1)}n^{O(1)}$ since it is of constant
dimension. For each such $Q$, the number of subsets of $Q$ of size at most $d$ is
at most  ${|Q| \choose \leq d} \leq |Q|^d = d^{O(d)}$. Thus, the number of choices of $X'$ is at most 
$d^{O(d)}n^{O(1)}$. 

Let $m=|X\setminus X'|$ and order the elements of $X \setminus X'$ as $x_1<\dots<x_m$. Consider the nested sequence $X_0 \subset X_1 \subset \dots \subset X_{m}$ of sets where $X_0=X'$ and $X_{i+1}=X_i \cup \{x_{i+1}\}$, so $X_m=X$. If $|\Sigma(X_{i+1})|<2|\Sigma(X_i)|$, then $x_{i+1}$ is in the difference set
$\Sigma(X_{i})-\Sigma(X_{i})$ and there are at most $|\Sigma(X_i)|^2 \leq |\Sigma(X)|^2 \leq 
2^{2s}d^4 \leq d^{6}$ such choices for $x_{i+1}$. Otherwise,
$|\Sigma(X_{i+1})|=2|\Sigma(X_i)|$ and there are at most $n$ choices for
$x_{i+1}$.

Note that, by Lemma \ref{lem:Sigmabasic}, we have $|\Sigma(X')| \geq {|X'|+1 \choose 2}+1 \geq d^2/4$. 
Suppose there are $t$ elements $x_{i+1}$ where $|\Sigma(X_{i+1})|=2|\Sigma(X_i)|$.
We have $t  \leq s+2$ as $|\Sigma(X_0)|=|\Sigma(X')| \geq d^2/4$ and
$2^sd^2 \geq |\Sigma(X)|=|\Sigma(X_m)| \geq 2^t|\Sigma(X_0)|$. 

Thus, after selecting the at most $s+2$ indices $i$ for which  $|\Sigma(X_{i+1})|=2|\Sigma(X_i)|$, we have that the number of possible $d$-sets $X \subset [n]$ such that $|\Sigma(X)| \leq 2^sd^2$ is at most 
$$d^{O(d)}n^{O(1)}{m \choose \leq s+2}n^{s+2}d^{6m}=d^{O(d)}n^{O(s)}.$$ This completes the proof of the
lemma. 
\end{proof}

\vspace{3mm}

{\bf Proof of Theorem \ref{thm:Hegy}:}
Let $A$ be the random set formed by choosing each element independently with probability $\delta$ and let $E$ be the event that $A$ contains a Hilbert cube of dimension $d = c \sqrt{\log n}$. We wish to show that the probability of the event $E$ tends to zero with $n$. 

Let $m_t$ denote the number of $d$-sets $X \subset [n]$ with $|\Sigma(X)|=t$. We have 
\begin{eqnarray*}
\mathbb{P}[E] & \leq & n \sum_{X} \delta^{|\Sigma(X)|}\\
& = & n \sum_{t} m_t\delta^{t},
\\ 
& = & n \sum_{t < d^3 } m_t\delta^{t}+ n \sum_{t \geq d^3} m_t\delta^{t}.
\end{eqnarray*} 

Note that the extra factor of $n$ comes from summing over all possible first coordinates $x_0$ for the Hilbert cube. The total number of choices for $X$ is at most $n^d$ by choosing the basis elements arbitrarily. Hence,  the second term is at most $n \cdot n^d \cdot \delta^{d^3} = o(1)$,
where we use $d=c\sqrt{\log n}$ and $c$ is a sufficiently large constant depending on $\delta$. 

We next bound the first term above. Let $a_s=2^{s-3}d^2$ and $n_s=\sum_{t=a_s}^{a_{s+1}-1}m_t$. 
We use Lemma \ref{hilbertcount} which gives $n_s \leq n^{O(s)}d^{O(d)}$. 
Hence, using a dyadic partition of the sum, the first term is at most 
$$n\sum_{s=1}^{2+\log d} n_s\delta^{2^{s-3}d^2} \leq n\sum_{s=1}^{2+\log d} n^{O(s)}d^{O(d)}\delta^{2^{s-3}d^2} =o(1),$$
where the last estimate uses $d=c\sqrt{\log n}$ and $c$ is a sufficiently large constant depending on $\delta$ to ensure that each of the terms in the sum is $o(n^{-2})$. 

It therefore follows that $\mathbb{P}[E]$ tends to zero as $n$ tends to infinity, completing the proof.
{\hfill$\Box$\medskip}

By considering a random $r$-coloring of $[n]$ and carefully checking the dependence of the constant factor on $\delta$ in the above proof, we have the following corollary. 

\begin{corollary}
There exists a constant $c$ such that
\[h(d,r) \geq r^{c d^2}.\]
\end{corollary}

We will not be able to improve this further unless we can improve the bound on the van der Waerden number $W(k,r)$, the smallest number which guarantees a $k$-term AP in any $r$-coloring of $[W(k,r)]$. Indeed, the lower bound on $W(k,2)$ is exponential in $k$ and, since a $d^2$-term AP contains a $d$-cube, $W(d^2,r) \geq h(d,r)$ . It seems plausible that $W(d^2,2)$ and hence $h(d,2)$ grow as exponentials in $d^2$.

\end{document}